\documentclass[10pt]{amsart}
\usepackage{amsmath,amssymb,amsthm,mathrsfs,booktabs,a4wide}
\usepackage[utf8]{inputenc}
\usepackage{graphicx}
\usepackage[algoruled,lined]{algorithm2e}
\title[A new extrapolation method for weak approximation schemes]{A new extrapolation method for weak approximation schemes with applications}
\author{Kojiro Oshima}
\address{Graduate School of Information Science and Technology, The University of Tokyo, 7-3-1 Hongo, Bunkyo-ku, Tokyo 113-8656, Japan}
\email{kojiro.oshima\@@gmail.com}
\author{Josef Teichmann}
\address{D-MATH, ETH Z\"urich, R\"amistrasse 101, 8092 Z\"urich, Switzerland}
\email{josef.teichmann\@@math.ethz.ch}
\author{Dejan Velu\v{s}\v{c}ek}
\address{University of Ljubljana, Faculty of Mathematics and Physics, Department of Mathematics, Jadranska 19, SI-1000 Ljubljana, SLovenia}
\email{dejan.veluscek\@@fmf.uni-lj.si}
\theoremstyle{plain}
\newtheorem{theo}{Theorem}
\newtheorem{prop}{Proposition}
\newtheorem{lem}{Lemma}
\newtheorem{cor}{Corollary}

\theoremstyle{definition}
\newtheorem{defi}{Definition}
\theoremstyle{remark}
\newtheorem{ex}{Example}

\newtheorem{rem}{Remark}

\providecommand{\N}{\mathbf{N}}
\providecommand{\R}{\mathbf{R}}
\begin{document}

\begin{abstract}
We review Fujiwara's scheme, a sixth order weak approximation scheme for the numerical approximation of SDEs, and embed it into a general method to construct weak approximation schemes of order $ 2m $ for $ m \in \mathbf{N} $. Those schemes cannot be seen as cubature schemes, but rather as universal ways how to extrapolate from a lower order weak approximation scheme, namely the Ninomiya-Victoir scheme, for higher orders.
\end{abstract}
\maketitle

\section{Introduction}

The Ninomiya-Victoir scheme for the weak approximation of solutions of stochastic differential equations can be described in the following framework: let $ (\Omega, \mathscr{F}, P) $ be a probability space and let $ \{B^1_t,\ldots ,B^d_t\}_{t \in \mathbf{R_+}} $ be a $ d $-dimensional standard Brownian motion. Define $ B^0_t:=t $ and $ B_t := (B_t^0, B^1_t, \ldots ,B^d_t) $. We consider stochastic differential equations driven by the Brownian motion $ \{B_t\}_{t \in \mathbf{R}_+} $
\begin{equation} \label{sde}
X(t,x) = x + \displaystyle\sum_{i=0}^d{\int_0^t{V_i(X(s,x))\circ dB^i_s}} 
\end{equation}
where $ x $ is in $ \mathbf{R}^N $, $ V_i \in C_b^\infty(\mathbf{R}^N;\mathbf{R}^N) $ and $ \circ $ stands for Stratonovich integral. We associate for later use the following simple stochastic differential equations to equation \eqref{sde}
\begin{equation} \label{simple_sde}
X^{(i)}(t,x) = x + \int_0^t{V_i(X^{(i)}(s,x))\circ dB^i_s}. 
\end{equation}
Let $ \{P_t\}_{t \in \mathbf{R}_+}  $  and $ \{P^{(i)}_t\}_{t \in \mathbf{R}_+} $ be the associated heat semigroups on $ C_b^\infty(\mathbf{R}^d) $ such that $P_tf(x) := E[f(X(t,x))] $ for $ t \geq 0 $, and $P^{(i)}_tf(x) := E[f(X^{(i)}(t,x))] $ for $ t \geq 0 $. Notice here that the equation associated to the index $0$ is a pure drift equation, the semigroup a transport semigroup. Denote furthermore by
\begin{align*}
\mathcal{A} :&= V_0 + \frac{1}{2}\sum_{i=1}^d V_i^2, \\
\overrightarrow{Q_t^{[\theta]}} :&= \left( P^{(0)}_{t/\theta} \circ \cdots \circ P^{(d)}_{t/\theta}\right)^\theta, \\ 
\overleftarrow{Q_t^{[\theta]}} :&= \left( P^{(d)}_{t/\theta} \circ \cdots \circ P^{(0)}_{t/\theta}\right)^\theta, \\
Q_t^{[\theta]} :&= \frac{1}{2}(\overrightarrow{Q_t^{[\theta]}} + \overleftarrow{Q_t^{[\theta]}}). 
\end{align*}
the generator of the diffusion process \eqref{sde}, two ordered products of (semi-)flows with generators $ V_0 $ and $ V_i^2 $ and the average of the two ordered products $ Q^{[\theta]} $. Then we have the well-known short time asymptotics, formulated in the language of $k$-norms (see Definition \ref{k_norms})
$$
| P_t g(x) - Q_t^{[\theta]} g(x) | \leq C t^3 \| g \|_{6 \theta (d+1) },
$$
as $ t \to 0 $, leading -- by iteration -- to the Ninomiya-Victoir scheme. Indeed, when we define $n$-fold iteration of the operator $Q_{\frac{T}{n}}^{[\theta]}$
$$
Q^{[\theta]}_{T,n} = Q_{\frac{T}{n}}^{[\theta]} \circ \cdots \circ Q_{\frac{T}{n}}^{[\theta]},
$$
we obtain a scheme of weak approximation order $r=2$, i.e.,
$$
| P_T g(x) - Q_{T,n}^{[\theta]} g(x) | \leq \frac{C}{n^2} \| g \|_{6 \theta (d+1) }.
$$
Let us define formally weak approximations of $ P_T $ for a some fixed, finite $ T \in \mathbf{R}_+ $ of weak approximation order $r$. 

\begin{defi}[scheme of weak approximation order $ r $]
A family of linear operators $ \{Q_{T,n}\}_{n \in \mathbf{N}} $ on $ C_b^\infty(\mathbf{R}^d) $, continuous with respect to the supremum norm topology, is called a \textit{scheme of weak approximation order $ r $} if there exists $ C >0 $ and some number $ k \ge 0 $ such that 
\begin{equation}
|P_Tf(x) - Q_{T,n}f(x)| \le \frac{C}{n^r} \| f \|_k
\end{equation}
for all $ x \in \mathbf{R}^N $ and  for all $ f \in C^\infty_b(\mathbf{R}^d) $.
\end{defi}

Notice that the operator $ Q_{T,n} $ is only supposed to be linear and continuous with respect to the supremum norm topology on the set of $C^\infty_b$-function, but not necessarily of sub-Markovian type. This means in particular that classical (Romberg-)extrapolations belong to this class.

In \cite{Fujiwara} T.~Fujiwara constructs a sixth order scheme for smooth functions $C_b^\infty(\mathbf{R}^N) $ which consists of a linear combination of the previously described Ninomiya-Victoir scheme. Through the linear combination T.~Fujiwara can ``extrapolate'' the weak approximation order to $ r=6 $. In this paper, we  define \textit{generalized Fujiwara schemes of order $r=2m$} including the scheme in \cite{Fujiwara} by refining Fujiwara's technique to prove the convergence order and construct versions of weak approximation order $ r=2m $ for $ m \in \mathbf{N} $. We finally obtain the following Theorem \ref{theo:final_result}, whose proof can be found in Section \ref{sec:scheme}, notations can be found in the subsequent sections:

Let $ \{q_n\}_{n\in\mathbf{N}} $ be a generalized Fujiwara scheme of order $2m$, then 
$$ Q_{T,n} := \sum_{i=1}^mf_{\theta_i}(Q_{\frac{T}{n}}^{[\theta_i]})^n $$
for $ n \geq 0 $ is a scheme of weak approximation of order $ 2m $, where a choice of $ k $ is given by
$$
k = 2(2m + 1)(d+1) \sum_{i=1}^m \theta_i,
$$
that means
$$
| P_T \, g(x) - Q_{T,n} \, g(x) | \leq \frac{C}{n^{2m}} \| g \|_{k}
$$
for test functions $ g \in C^\infty_b (\mathbb{R}^N) $.

The remainder of the article is organized as follows: in Section 2 we introduce all algebraic prerequisities, in Section 3 we show the main algebraic result of this article, which is then applied in Section 4 to prove the existence of generalized Fujiwara schemes. In Section 5 we provide an implementation result, where the results can be compared to \cite{ninnin:2009}. The appendix is devoted to an original proof of Fujiwara's basic algebraic result.

\section{Algebraic prerequisites and their relation to weak approximation}

Let $ A $ be a set whose elements are $ a_0, \ldots, a_d $. We call $ A $ an alphabet and $ a_0, \ldots, a_d $ letters. A word in alphabet $ A $ is a finite sequence of letters. Let $ 1 $ be a empty word and $ A^* $ a set of words including $ 1 $. If we impose a total ordering on $A$, then $A^*$ together with word concatenation and lexicographic ordering becomes an ordered unital semigroup. Let $ \mathbf{R}\langle A \rangle $ be a set of noncommutative polynomials on $ A^* $ over $ \mathbf{R} $ i.e. a set of $ \mathbf{R} $--linear combinations of elements of $ A^* $ and let $ \mathbf{R} \langle \langle A \rangle \rangle $ be a set of noncommutative series of elements of $ A^* $ with coefficients in $\mathbf{R} $, i.e. a set of functions $f\colon A^* \to \mathbf{R}$ with well ordered support. Using componentwise addition and multiplication, which is induced by word concatenation, makes $ \mathbf{R}\langle\langle A^* \rangle\rangle $ a $\mathbf{R}$--algebra (see \cite{cohn} for
  more details). The degree of a monomial is a number of letters contained in the monomial and the degree of a noncommutative polynomial and a noncommutative series are the maximum degree of monomials contained in them. Let $ \mathbf{R}\langle A \rangle_m $ and $ \mathbf{R}\langle A \rangle_{\le m} $ be the set of homogeneous polynomials of the degree $m$ and the set of polynomials of the degree less or equal to $ m $ respectively. Define $ \mathbf{R} \langle \langle A \rangle \rangle_m $ and $ \mathbf{R}\langle \langle A \rangle \rangle_{\le m} $ in the same manner. Since every $u\in \mathbf{R}\langle \langle A \rangle \rangle$ has a well ordered support, we can define $\mathbf{R}\langle \langle A \rangle \rangle_{>m}=\{u\in \mathbf{R}\langle \langle A \rangle \rangle\rvert\, \deg(\inf(\mathrm{supp}(u)))>m \}$ and $\mathbf{R}\langle \langle A \rangle \rangle_{\geq m}=\{u\in \mathbf{R}\langle \langle A \rangle \rangle\rvert\, \deg(\inf(\mathrm{supp}(u)))\geq m \}$ and it is easy to see that $ \mathbf{R}\langle \langle A \rangle \rangle_{> m}$ and $\mathbf{R}\langle \langle A \rangle \rangle_{\geq m} $ are double sided ideals in algebra $ \mathbf{R}\langle \langle A \rangle \rangle$. Let $j_m $ and $ j_{\le m} $ be the natural surjective maps from $ \mathbf{R}\langle \langle A \rangle \rangle $ onto $ \mathbf{R}\langle\langle A \rangle \rangle_m $ and $ \mathbf{R}\langle\langle A \rangle \rangle_{\le m} $ respectively. 

Since every subset of $A^*$ has a least element regarding lexicographical ordering, we have $\mathbf{R} \langle\langle A\rangle\rangle=\mathbf{R}^{A^*}$. The set $A^*$ is countable, therefore taking metric topology in $\mathbf{R}$ makes $\mathbf{R}^{A^*}$ with induced product topology into a Polish space. Hence, we can consider its Borel $\sigma$-algebra $\mathcal{B}(\mathbf{R}\langle \langle A \rangle \rangle)$, $\mathbf{R}\langle \langle A \rangle \rangle$--valued random variables and expectations, and other notions as usual.

For $ u \in \mathbf{R}\langle \langle A \rangle \rangle $ we define the exponential map
$$
\exp{(u)} := \sum_{n \ge  0}\frac{u^n}{n!},
$$
and for $ u \in \mathbf{R}\langle \langle A \rangle \rangle $ with vanishing constant term, we define the logarithm,
$$
\log{(1+u)}:= \sum_{n \ge 1 }\frac{(-1)^{n-1}}{n}u^n. 
$$
It is easy to check that 
\begin{align}
\log{(\exp{(u)})} &= u, \\
\exp{(\log{(u)})} &= u,
\end{align}
on the respective domains. For $ \theta \in \mathbf{N} $ define 
\begin{align*}
p :&= \exp{(\sum_{i=0}^da_i)}, \\
\overrightarrow{q^{[\theta]}} :&= \left( \exp{(\frac{1}{\theta}a_0)} \cdots \exp{(\frac{1}{\theta}a_d)}\right)^\theta, \\ 
\overleftarrow{q^{[\theta]}} :&= \left( \exp{(\frac{1}{\theta}a_d)} \cdots \exp{(\frac{1}{\theta}a_0)}\right)^\theta, \\
q^{[\theta]} :&= \frac{1}{2}(\overrightarrow{q^{[\theta]}} + \overleftarrow{q^{[\theta]}}). 
\end{align*}

Let us make the substitution, which is the heart of the transfer from algebra to numerical schemes, $ a_0 = V_0 , a_1 = V_1^2/2, \ldots , a_d = V_d^2/2 $ formally correct. Let $ B $ be another alphabet including $ v_0, v_1, \ldots , v_d $ and set $ B^*, \mathbf{R} \langle B \rangle , \ldots , $ in the same manner. For all $ t \in \mathbf{R}_+ $ define an algebra homomorphism $ \Psi_t : \mathbf{R} \langle \langle A \rangle \rangle \longrightarrow   \mathbf{R} \langle \langle B \rangle \rangle $ by setting
\begin{align}
\Psi_t (a_0) &:= tv_0, \\
\Psi_t (a_i) &:= tv_i^2/2.
\end{align}
for all $ i \in \{ 1,\ldots , d \} $. 

Define next an algebra homomorphism $\Phi\colon \mathbf{R} \langle B\rangle \to C_b^\infty(\mathbf{R}^N;\mathbf{R}^N)$ by setting
\begin{align}
\Phi (v_i) &= V_i.
\end{align} 
Let $D=\{ \sum_{w\in B^*} a_w w\rvert \sum_{w\in B^*} a_w \Phi(w) \text{ is well defined }\}$. Clearly, $\mathbf{R} \langle B\rangle\subset D$ and $D$ is a $\mathbf{R}$--subalgebra of $\mathbf{R} \langle\langle B\rangle\rangle$. The homomorphism $\Phi$ can then be uniquely extended to an $\mathbf{R}$--algebra homomorphism $\Phi\colon D \to C_b^\infty(\mathbf{R}^N;\mathbf{R}^N) $.

The algebra of non-commutative words plays a major role in the analysis of weak approximation schemes due to the following well-known asymptotic expansion theorem, which allows to approximate the truncated exponential series in $ \mathcal{A} $ by other simpler expressions.
\begin{theo}\label{theo:taylor}
For all function $ f \in C^\infty_{b}(\mathbf{R}^N)$, $ x \in \mathbf{R}^N $ and $ n \in \mathbf{N} $, it holds that
\begin{equation}
P_tf(x) = \sum_{k=0}^n{\frac{t^k}{k!}\mathcal{A}^kf(x)} + \mathcal{O}(t^{n+1}) = \Phi (\Psi_t (j_{\leq n} p))f(x) + \mathcal{O}(t^{n+1}).
\end{equation}
as $ t \to 0 $.
\end{theo}
\begin{proof}
See \cite{IW}.
\end{proof}

Hence we can, e.g., express the generator $ \mathcal{A} $ of the diffusion process \eqref{sde} by
$$
\Phi(\Psi_1(a_0+\ldots+a_d)) = \mathcal{A}, 
$$
in particular we obtain the following crucial asymptotic formulas,
$$
\Phi \Psi_{t} (j_{\leq n}(\exp(a_i))) = P^{(i)}_t + \mathcal{O}(t^{n+1})
$$
as $ t \to 0 $ and $ i=0,\ldots,d $ again due to Theorem \ref{theo:taylor}.

To be more precise on the goal of our paper, Theorem \ref{theo:taylor} also means that if we approximated $ p $ by linear combinations of $ {(q^{[\theta]})}^n $ up to a certain degree $ 2m-1 $ within the algebra $ \mathbf{R} \langle A \rangle $ such that the remainder term is of order $ \mathcal{O}(\frac{1}{n^{2m}}) $, then $ P_tf(x) $ could be approximated by linear combinations of $ \Phi  (\Psi_t({(q^{[\theta]})}^n) )f(x) $ in a weak sense of order $ 2m$.

Notice that the letters $ a_i $ correspond to squares of vector fields under $ \Phi \circ \Psi_t $, hence one has to work out the correspondence to exponentials of first order terms, too. The next lemma shows how to relate thoes linear semi-flows of PDEs $ P^{(i)}_t $ to non-linear flows of ODEs $ \operatorname{Fl}^{V_i}_t (x) $ up to a certain degree $ m $, namely by replacing the normal random variable $ Z $ by a random variable taking finitely many values and sharing moments up to order $ 2m $. This finally means that we can approximate $ q^{[\theta]} $ by convex combinations of exponentials of first degree terms, i.e.~$a_0,\ldots,a_d$ leading to weak approximation schemes.

\begin{lem}
For all $ i \in \{1,\dots,d\} $ we have that
\begin{equation}
E[\exp{(B^i_t v_i)}] = \exp{(t\frac{v_i^2}{2})}
\end{equation}
holds true. This formula also holds true under the homomorphism $ \Phi \circ \Psi_1 $, i.e.,
$$
E[f(\operatorname{Fl}^{V_i}_{B^i_t}(x))] = P^{(i)}_t f (x)
$$
for test functions $ f $ and $ x \in \R^N $.
\end{lem}
\begin{proof}
Proof by applying the Fourier transform of Brownian motion and classical subordination results.
\end{proof}

\section{How to approximate $p$ by $q$?}

An alternative proof of this result can be found in the appendix:

\begin{lem}[ {\cite{Fujiwara}} Lemma 2.1]\label{lem:fujiwara}
We have
\begin{equation}\label{eq:expansion}
\log{\overleftarrow{q^{[1]}}} = \sum_{i=1}^\infty{(-1)^{i+1}j_i(\log{\overrightarrow{q^{[1]}}})}. 
\end{equation}
\end{lem}

\begin{prop}[\cite{Fujiwara} Proposition 2.2]\label{prop:fujiwara}
There exists $c_i \in \mathbf{R}\langle \langle A \rangle \rangle_{\ge 2i+1} $ such that for all $ \theta \in \mathbf{N} $,
\begin{equation*}
q^{[\theta]} = p + \sum_{i=1}^\infty\frac{c_i}{\theta^{2i}},
\end{equation*}
holds.
\end{prop}

\begin{cor}
Let $ q $ be a linear combination of $ q^{[\theta]}$ for some $\theta\in\mathbf{N}$. If there exists $ n \in \mathbf{N} $ such that $ j_{\le 2n-1}(q) = j_{\le 2n-1}(p) $, then  $ j_{\le 2n}(q) = j_{\le 2n}(p) $.
\end{cor} 
\begin{proof}
For all $ \theta \in \mathbf{N} $, $ j_{\le 2}(q^{[\theta]}) = j_{\le 2}(p) $ holds. Hence, the case $ n=1 $ is clear. Suppose $n\geq 2$ and $ j_{\le 2n-1}(q) = j_{\le 2n-1}(p) $. Since $q=\sum_{j=1}^k \alpha_j q^{[\theta_j]}$ for some $\theta_i\in\mathbf{N}$ and since $ j_{\le 2}(q^{[\theta]}) = j_{\le 2}(p) $ for all $\theta \in \mathbf{N} $, it follows $\sum_{j=1}^k \alpha_j=1$. According to Proposition \ref{prop:fujiwara} 
\[
 	q=p+\sum_{i=1}^\infty c_i (\sum_{j=1}^k \alpha_j \frac{1}{\theta_j^{2i}})
\]
for some $c_i\in \mathbf{R}\langle\langle A\rangle\rangle_{\geq 2i+1}$. Since $ j_{\le 2n-1}(q) = j_{\le 2n-1}(p) $, we have 
\[
	\sum_{j=1}^k \alpha_j \frac{1}{\theta_j^{2i}}=0
\]
for all $i=1,\dots, n-1$. Then
\[
	q-p=\sum_{i=n}^\infty c_i (\sum_{j=1}^k \alpha_j \frac{1}{\theta_j^{2i}}),\quad \text{ where } c_n\in \mathbf{R}\langle\langle A\rangle\rangle_{\geq 2n+1}, 
\]
which proves the corollary.
\end{proof}

Set 
\begin{equation*}
A := \begin{bmatrix}1 & \cdots & 1 \\
                    1/\theta_1^2 & \cdots & 1/\theta_m^2 \\
                    \vdots & \ddots & \vdots \\
                    1/\theta_1^{2(m-1)} & \cdots & 1/\theta_m^{2(m-1)}
      \end{bmatrix}.
\end{equation*}

\begin{cor} \label{cor:order}
\begin{equation}
j_{\le 2m}\left( \left(A^{-1}\begin{bmatrix}1 \\ 0 \\ \vdots \\ 0 \end{bmatrix}\right)^T \begin{bmatrix} q^{[\theta_1]}-p \\ \vdots \\ q^{[\theta_m]}-p \end{bmatrix} \right) =0
\end{equation}
holds.
\end{cor}

\begin{cor} \label{cor:key}
For all $ l \in \{1,\ldots,m-1\} $,
\begin{equation}
\left(A^{-1} \begin{bmatrix} 1 \\ 0 \\ \vdots \\ 0 \end{bmatrix}\right)^T\begin{bmatrix} \frac{1}{\theta_1^{2l}} \\ \vdots \\ \frac{1}{\theta_m^{2l}} \end{bmatrix}=0
\end{equation}
\end{cor}

\section{Generalized Fujiwara scheme and its property}\label{sec:scheme}

\begin{defi}[Generalized Fujiwara scheme]
A family of series,
\begin{equation}
\big\{q_n:=\sum_{i=1}^mf_{\theta_i}(q^{[\theta_i]})^n \big\}_{n\in\mathbf{N}}
\end{equation}
is called a \textit{generalized Fujiwara scheme} of order $2m$ if 
\begin{align*}
f &= \begin{bmatrix}f_{\theta_1} & \cdots & f_{\theta_m} \end{bmatrix}^T \\
  &=A^{-1} \begin{bmatrix} 1 \\ 0 \\ \vdots \\ 0 \end{bmatrix},
\end{align*}
holds. 
\end{defi}

A straightforward calculation involving induction gives the following connection concerning the powers of series in $\mathbf{R}\langle \langle A \rangle \rangle $. Notice that we split the product $ q^n - p^n $ into telescoping summands, where one, two up to $ m $ terms of the form $ q-p $ appear. 
\begin{prop}\label{prop:decomposition}
For $p,q \in \mathbf{R}\langle \langle A \rangle \rangle $ and $ 2 \le m \le n $, we have
\begin{align*}
&q^n - p^n  = \\
&=\sum_{k=0}^{n-1}p^k(q-p)p^{n-k-1} +\\
+ \sum_{l=2}^{m-1}&\big( \sum_{k_l=l-1}^{n-1}\sum_{k_{l-1}=l-2}^{k_l-1}\cdots \sum_{k_2=1}^{k_3-1}\sum_{k_{1}=0}^{k_2-1} p^{k_1}(q-p)p^{(k_2-k_1-1)}(q-p) \times \cdots \\ 
&\times p^{(k_l-k_{l-1}-1)}(q-p)p^{n-k_l-1} \bigr) + \\
+ \sum_{k_m=m-1}^{n-1}&\sum_{k_{m-1}=m-2}^{k_m-1}\cdots \sum_{k_2=1}^{k_3-1}\sum_{k_{1}=0}^{k_2-1} q^{k_1}(q-p)p^{k_2-k_1-1}(q-p) \times \cdots\\ 
&\times p^{k_m-k_{m-1}-1}(q-p)p^{n-k_m-1}. 
\end{align*}
In particular for $ m=1 $,
\begin{equation*}
q^n-p^n = \sum_{k=0}^{n-1}q^k(q-p)p^{n-k-1}
\end{equation*}
holds true.
\end{prop}

\begin{lem}
For $ z_1, z_2 \in \mathbf{R}\langle \langle A \rangle \rangle $, if $ j_{\le l}(z_1) =0 $ and $ j_{\le m}(z_2)=0 $, then $j_{\le l+m+1}(z_1z_2) =0 $.
\end{lem}
\begin{proof}
By the assumption,  monomials with the lowest degree contained in $ z_1 $ and $ z_2 $ are of the degree $ l+1 $ and $ m+1 $. Then, monomial with the lowest degree contained in $ z_1z_2 $ has the degree $ l+m+2 $. Hence $ j_{\le l+m+1}(z_1z_2) =0 $.
\end{proof}

\begin{cor}\label{cor:substitute}
For $z_1,z_2,z_3 \in \mathbf{R}\langle \langle A \rangle \rangle $, if $ j_{\le l}(z_1) = j_{\le l}(z_2) $ and $ j_{\le m}(z_3)=0 $, then $ j_{\le l+m+1}(z_1z_3) = j_{\le l+m+1}(z_2z_3)$.
\end{cor}

\begin{cor}\label{cor:multiple}
For $ z \in \mathbf{R}\langle \langle A \rangle \rangle $, if $ j_{\le l}(z) =0 $, then $ j_{\le ml+m-1}(z^m)=0 $.
\end{cor}

\begin{theo}\label{theo:critical}
If a series,
\begin{equation}
q_n:=\sum_{i=1}^mf_{\theta_i}(q^{[\theta_i]})^n,
\end{equation}
is a $2m$-th order generalized Fujiwara scheme, then for all $ l \in \{ 2,\ldots,m-1 \} $, 
\begin{equation}
j_{\le 2m+l-1}(\sum_{i=1}^mf_{\theta_i}(q^{[\theta_i]} - p)^l) = 0
\end{equation}
holds true.
\end{theo}
\begin{proof}
Fix $ l \in \{2,\ldots,m-1\} $. By Proposition \ref{prop:fujiwara}, 
\begin{equation}
	\sum_{i=1}^m f_{\theta_i}(q^{[\theta_i]}-p)^l=\sum_{i=1}^mf_{\theta_i}\sum_{i_1,\ldots,i_l=1}^\infty\frac{c_{i_1}\cdots c_{i_l}}{\theta_i^{2(i_1 + \cdots i_l)}}
\end{equation}
holds. It is easy to see that 
\[c_{i_1} \cdots c_{i_l} \in \mathbf{R}\langle \langle A \rangle \rangle_{\ge 2(i_1+ \cdots i_l)+l}.\]
Hence, we have
\begin{align*}
&j_{\le 2m+l-1}\left(\sum_{i_1,\ldots,i_l=1}^\infty \frac{c_{i_1}\cdots c_{i_l}}{\theta_i^{2(i_1+\cdots + i_l)}} \right) \\
=& j_{\le 2m+l-1}\left(\sum_{\substack{i_1, \ldots, i_l \ge 1 \\ i_1+\cdots + i_l \le m-1}}\frac{c_{i_1}\cdots c_{i_l}}{\theta^{2(i_1+\cdots i_l)}}\right) \\
=& j_{\le 2m+l-1}\left(\sum_{k=l}^{m-1}\sum_{\substack{i_1, \ldots, i_l \ge 1 \\ i_1+\cdots + i_l = k}}\frac{c_{i_1}\cdots c_{i_l}}{\theta^{2k}}\right).
\end{align*}
Thus, we have 
\begin{align*}
&j_{\le 2m+l-1}(\sum_{i=1}^mf_{\theta_i}(q^{[\theta_i]}-p)^l) \\
=&j_{\le 2m+l-1}(\sum_{k=l}^{m-1}\sum_{\substack{i_1, \ldots, i_l \ge 1 \\ i_1+\cdots + i_l = k}}\sum_{i=1}^mf_{\theta_i}\frac{1}{\theta_i^{2k}}c_{i_1} \cdots c_{i_l}) \\
=& j_{\le 2m+l-1}(\sum_{k=l}^{m-1}\sum_{\substack{i_1, \ldots, i_l \ge 1 \\ i_1+\cdots + i_l = k}}f^T \begin{bmatrix} \frac{1}{\theta_1^{2k}} \\ \vdots \\ \frac{1}{\theta_m^{2k}} \end{bmatrix}c_{i_1}\cdots c_{i_l}) \\
=&0
\end{align*}
by Corollary \ref{cor:key}.
\end{proof}

\begin{defi}
  A generalized power series $a\in \R\langle\langle A\rangle\rangle$ is an element of $O(s)$ if for every $n,N\in\N$, $n\leq N$ there exists a uniform bound for all the coefficients of the terms of $\frac{1}{s}a$ with degree $k$, which satisfies $n\leq k\leq N$ as $s\to 0$.
\end{defi}

\begin{theo} \label{theo:connection}
If $ \{q_n\}_{n\in\mathbf{N}} $ is an $ m $-th order generalized Fujiwara scheme, then 
$$ \Psi_{1/n} q_n - \Psi_1 p =\sum_{l=1}^m \sum_{k_l=l-1}^{n-1} \sum_{k_{l-1}=l-2}^{k_l-1}\dots \sum_{k_2=1}^{k_3-1} \sum_{k_1=0}^{k_2-1}\Psi_{1/n} a_{l,(k_l,\dots,k_1)},
$$
where $j_{\leq 2m+l-1}(a_{l,(k_l,\dots,k_1)})=0$ for $ l=1,\ldots,m$, and $ \Psi_{1/n} q_n - \Psi_1 p \in O(\frac{1}{n^{2m}})$.
\end{theo}
\begin{proof}
Let $ m \ge 2 $. The case $ m=1 $ is trivial.
Let $ \{q_n:=\sum_{i=1}^mf_{\theta_i}(q^{[\theta_i]})^n \}_{n\in\mathbf{N}}$ be an $ m $-th order generalized Fujiwara scheme. Note that $\Psi_1 p=(\Psi_{1/n} p)^n$. Then by Proposition \ref{prop:decomposition}, we have, 
\begin{align*}
&\Psi_{1/n} q_n - \Psi_1 p  \\
&= \sum_{i=1}^mf_{\theta_i}\sum_{k=0}^{n-1}(\Psi_{1/n} p)^k(\Psi_{1/n} q^{[\theta_i]}-\Psi_{1/n} p)(\Psi_{1/n} p)^{n-k-1} \\
&+ \sum_{i=1}^mf_{\theta_i}\sum_{l=2}^{m-1}\sum_{k_l=l-1}^{n-1}\sum_{k_{l-1}=l-2}^{k_l-1}\cdots \sum_{k_2=1}^{k_3-1}\sum_{k_{1}=0}^{k_2-1} {(\Psi_{1/n} p)}^{k_1}(\Psi_{1/n} q^{[\theta_i]}-\Psi_{1/n} p){(\Psi_{1/n} p)}^{k_2-k_1-1} \times \cdots \\ & \times (\Psi_{1/n} q^{[\theta_i]}-\Psi_{1/n} p) 
{(\Psi_{1/n} p)}^{k_l-k_{l-1}-1} (\Psi_{1/n} q^{[\theta_i]}-\Psi_{1/n} p){(\Psi_{1/n} p)}^{n-k_l-1} \\
&+ \sum_{i=1}^mf_{\theta_i}\sum_{k_m=m-1}^{n-1}\sum_{k_{m-1}=m-2}^{k_m-1}\cdots \sum_{k_2=1}^{k_3-1}\sum_{k_{1}=0}^{k_2-1} {(\Psi_{1/n} q^{[\theta_i]})}^{k_1}(\Psi_{1/n} q^{[\theta_i]}-\Psi_{1/n} p){(\Psi_{1/n} p)}^{k_2-k_1-1} \times \cdots \\ 
& \times (\Psi_{1/n} q^{[\theta_i]}-\Psi_{1/n} p) {(\Psi_{1/n} p)}^{k_m-k_{m-1}-1} (\Psi_{1/n} q^{[\theta_i]}-\Psi_{1/n} p){(\Psi_{1/n} p)}^{n-k_m-1}.
\end{align*}
Set 
\[
  a_{1,(k_1)}=\sum_{i=1}^m f_{\theta_i} p^{k_1}(q^{[\theta_i]}-p) p^{n-k_1-1}.
\]
For $l\in \{2,\dots,m-1\}$ set
\[
  a_{l,(k_l,\dots,k_1)}=\sum_{i=1}^mf_{\theta_i} p^{k_1}(q^{[\theta_i]}-p) p^{k_2-k_1-1} (q^{[\theta_i]}- p) \cdots p^{k_l-k_{l-1}-1} (q^{[\theta_i]}-p) p^{n-k_l-1} 
\]
and for $l=m$ define
\[
  a_{m,(k_m,\dots,k_1)}=\sum_{i=1}^mf_{\theta_i} (q^{[\theta_i]})^{k_1}(q^{[\theta_i]}-p) p^{k_2-k_1-1} (q^{[\theta_i]}-p) \cdots p^{k_m-k_{m-1}-1} (q^{[\theta_i]}-p) p^{n-k_m-1}.
\]
In particular the summand $a_{1,(k_1)}$ can be written as
\begin{equation}
a_{1,(k_1)} =p^{k_1} \sum_{i=1}^mf_{\theta_i}(q^{[\theta_i]}-p)p^{n-k_1-1}.
\end{equation}
Let $c_i\in \mathbf{R}\langle\langle A\rangle\rangle_{\geq 2i+1}$, $i\in \mathbf{N}$ be as in Proposition \ref{prop:fujiwara}. By Theorem \ref{theo:critical}, $$ j_{\le 2m}(\sum_{i=1}^mf_{\theta_i}(q^{[\theta_i]}-p)) =0, $$ thus $j_{\le 2m}(a_{1,(k_1)})=0$. Also it holds that $$j_{2m+1}(\sum_{i=1}^mf_{\theta_i}(q^{[\theta_i]}-p))=j_{2m+1}(c_m) [1/\theta_1^{2m},\dots,1/\theta_m^{2m}] f. $$ 

By Corollary \ref{cor:order}, for all $ \theta \in \mathbf{N} $, $ j_{\le 2}(q^{[\theta]} - p )=0 $ holds. Thus, by Corollary \ref{cor:multiple}, Corollary \ref{cor:substitute} and Proposition \ref{prop:fujiwara}
\begin{equation}
j_{\le 3m-1}((q^{[\theta_i]})^{k_1}(q^{[\theta_i]}-p)p^{k_2-k_1-1}(q^{[\theta_i]}-p) \cdots p^{k_m-k_{m-1}-1}(q^{[\theta_i]}-p)p^{n-k_m-1}) =0
\end{equation} 
holds, hence $j_{\le 3m-1}(a_{m,(k_m,\dots,k_1)})=0$. Moreover, 
\[
 j_{3m}((q^{[\theta_i]})^{k_1}(q^{[\theta_i]}-p)p^{k_2-k_1-1}(q^{[\theta_i]}-p) \cdots p^{k_m-k_{m-1}-1}(q^{[\theta_i]}-p)p^{n-k_m-1})=(j_{3}(c_1))^m \frac{1}{\theta_i^{2m}}.
\]

Let $p_1,\dots,p_{l+1}\in \mathbf{R}\langle\langle A\rangle\rangle$ with property $j_0(p_i)=1$ for $i\in\{1,\dots,l+1\}$. By using similar arguments as in the proof of Theorem \ref{theo:critical}, we get
\begin{align*}
&j_{\le 2m+l-1}\Big(\sum_{i=1}^m f_{\theta_i}p_{1}(q^{[\theta_i ] }-p)p_{2}(q^{[\theta_i ]}-p) 
\cdots p_{l}(q^{[\theta_i ]}-p)p_{l+1}\Big)\\
&= j_{\le 2m+l-1}\Big(\sum_{k=l}^{m-1}\sum_{\substack{i_1, \ldots, i_l \ge 1 \\ i_1+\cdots + i_l = k}}f^T \begin{bmatrix} \frac{1}{\theta_1^{2k}} \\ \vdots \\ \frac{1}{\theta_m^{2k}} \end{bmatrix}p_1c_{i_1}p_2\cdots p_lc_{i_l}p_{l+1}\Big) \\
&=0,
\end{align*}
and
\begin{align*}
 	j_{2m+l}&\Big(\sum_{i=1}^m f_{\theta_i}p_{1}(q^{[\theta_i]}-p)p_{2}(q^{[\theta_i]}-p) 
\cdots p_{l}(q^{[\theta_i]}-p)p_{(l+1)}\Big)\\
=& \sum_{\substack{i_1, \ldots, i_l \ge 1 \\ i_1+\cdots + i_l = m}} f^T \begin{bmatrix}
                                                                         \frac{1}{\theta_1^{2m}}\\
                                                                         \vdots\\
                                                                         \frac{1}{\theta_m^{2m}}
                                                                        \end{bmatrix}
                                                                        j_{2m+l}(c_{i_1}\cdots c_{i_l})
\end{align*}
for all $ l \in \{2,\ldots , m-1\} $. We conclude, that $j_{\le 2m+l-1}(a_{l,(k_l,\dots,k_1)})=0$.

It remains to prove that $ \Psi_{1/n} q_n - \Psi_1 p \in O(\frac{1}{n^{2m+1}})$. 

First, let us observe $a_{l,(k_l,\dots,k_1)}$ for $l\in\{1,\dots,m-1\}$. Choose $M\in\N\cup\{0\}$. As above we can write
\begin{align*}
 	j_{2m+l+M}&(a_{l,(k_l,\dots,k_1)})\\
 	=&j_{2m+l+M}\Big(\sum_{i=1}^mf_{\theta_i} p^{k_1}(q^{[\theta_i]}-p) p^{k_2-k_1-1} (q^{[\theta_i]}- p) \cdots p^{k_l-k_{l-1}-1} (q^{[\theta_i]}-p) p^{n-k_l-1}\Big)\\
=& \sum_{k=m}^{m+[M/2]}\sum_{\substack{i_1, \ldots, i_l \ge 1 \\ i_1+\cdots + i_l = k}} f^T \begin{bmatrix}
                                                                         \frac{1}{\theta_1^{2k}}\\
                                                                         \vdots\\
                                                                         \frac{1}{\theta_m^{2k}}
                                                                        \end{bmatrix}
                                                                        j_{2m+l+M}(p^{k_1}c_{i_1}p^{k_2-k_1-1}\cdots p^{k_l-k_{l-1}-1}c_{i_l}p^{n-k_l-1})
\end{align*}
The coefficient of the power $p^k$ of the term of degree $l$ is of the form $k^l c$, where $c$ is the coefficient of $p$ of the same degree.  Hence, the coefficient of the term of degree ${2m+l+M}$ of $$ p^{k_1}c_{i_1}p^{k_2-k_1-1}\cdots p^{k_l-k_{l-1}-1}c_{i_l}p^{n-k_l-1}$$ is a finite sum, namely $$\sum_{i\in I} k_1^{n_{1,i}} (k_2-k_1-1)^{n_{2,i}}\dots (n-k_l-1)^{n_{l+1,i}} b_i, $$ where $ b_i$ and the number of summands do not depend on $n$, and $n_{j,i}\in\N$, $\sum_{j=1}^{l+1} n_{j,i}\leq 2m+M - 2\sum_{k=1}^l i_k$. Let us denote $b'_{i,k}=[\frac{1}{\theta_1^{2k}}\dots \frac{1}{\theta_m^{2k}}]\, f\, b_i$. Thus, the coefficient of a term of degree $ 2m+l+M $ of
\[
\Psi_{1/n} \sum_{k_l=l-1}^{n-1} \sum_{k_{l-1}=l-2}^{k_l-1}\dots \sum_{k_2=1}^{k_3-1} \sum_{k_1=0}^{k_2-1} a_{l,(k_l,\dots,k_1)}
\]
has the following upper bound

\begin{align*}
  \frac{1}{n^{2m+M+l}}&\Big\lvert \sum_{k_l=l-1}^{n-1} \sum_{k_{l-1}=l-2}^{k_l-1}\dots \sum_{k_2=1}^{k_3-1} \sum_{k_1=0}^{k_2-1}\\
  &\sum_{k=m}^{m+[M/2]}\sum_{\substack{i_1, \ldots, i_l \ge 1 \\ i_1+\cdots + i_l = k}} \sum_{i\in I} k_1^{n_{1,i}} (k_2-k_1-1)^{n_{2,i}}\dots (n-k_l-1)^{n_{l+1,i}} b'_{i,k}\Big\rvert\\
  & \leq \frac{1}{n^{2m+M+l}} \sum_{k_l=l-1}^{n-1} \sum_{k_{l-1}=l-2}^{k_l-1}\dots \sum_{k_2=1}^{k_3-1} \sum_{k_1=0}^{k_2-1}\\
  &\sum_{k=m}^{m+[M/2]}\sum_{\substack{i_1, \ldots, i_l \ge 1 \\ i_1+\cdots + i_l = k}} \sum_{i\in I} \lvert k_1^{n_{1,i}} (k_2-k_1-1)^{n_{2,i}}\dots (n-k_l-1)^{n_{l+1,i}} b'_{i,k} \rvert\\
  & \leq \sum_{k_l=l-1}^{n-1} \sum_{k_{l-1}=l-2}^{k_l-1}\dots \sum_{k_2=1}^{k_3-1} \sum_{k_1=0}^{k_2-1} \sum_{k=m}^{m+[M/2]}\sum_{\substack{i_1, \ldots, i_l \ge 1 \\ i_1+\cdots + i_l = k}}\sum_{i\in I} \frac{1}{n^{l+2\sum_{k=1}^l i_k}} \lvert b'_{i,k}\rvert.
\end{align*}
Since the number of terms of 
\[
  j_{2m+l+M}(\Psi_{1/n} \sum_{k_l=l-1}^{n-1} \sum_{k_{l-1}=l-2}^{k_l-1}\dots \sum_{k_2=1}^{k_3-1} \sum_{k_1=0}^{k_2-1} a_{l,(k_l,\dots,k_1)})
\]
is finite and their number does not depend on $n$, there exists a uniform bound for all of the coefficients, which proves our assertion for $l<m$.

Let us observe the $a_{m,(k_m,\dots,k_1)}$. As above
\begin{align*}
 	j_{2m+l+M}&(a_{m,(k_m,\dots,k_1)})\\
 	=&j_{2m+l+M}\Big(\sum_{i=1}^mf_{\theta_i} (q^{[\theta_i]})^{k_1}(q^{[\theta_i]}-p) p^{k_2-k_1-1} (q^{[\theta_i]}- p) \cdots p^{k_m-k_{m-1}-1} (q^{[\theta_i]}-p) p^{n-k_m-1}\Big)\\
=& \sum_{k=m}^{m+[M/2]}\sum_{\substack{i_1, \ldots, i_m \ge 1 \\ i_1+\cdots + i_m = k}} \sum_{i=1}^m f_{\theta_i} \frac{1}{\theta_i^{2k}} j_{3m+M}((q^{[\theta_i]})^{k_1}c_{i_1}p^{k_2-k_1-1}\cdots p^{k_m-k_{m-1}-1}c_{i_m}p^{n-k_m-1})
\end{align*}
Since $q^{[\theta]}$ is a convex combination of products of $\exp (\frac{1}{\theta} a_i)$ all its coefficients are positive and the sum of all coefficients at the terms which are derived from $a_0^{s_0}\cdots a_d^{s_d}$ by permutation of $a_i$'s is exactly the coefficient at the term $a_0^{s_0}\cdots a_d^{s_d}$ of the $\exp (\sum_{i=0}^d a_i)$ in the (commutative) power series algebra, generated by $a_0,\dots,a_d$, with coefficients in $\R$ and the same goes for the coefficients of the power $(q^{[\theta]})^k$ and the coefficients of the commutative series $(\exp (\sum_{i=0}^d a_i))^k$. The rest of the argument goes as in the case of $l<m$, which gives us the upper bound of the coefficients at the terms of
\[
  j_{2m+l+M}(\Psi_{1/n} \sum_{k_m=m-1}^{n-1} \sum_{k_{m-1}=m-2}^{k_m-1}\dots \sum_{k_2=1}^{k_3-1} \sum_{k_1=0}^{k_2-1} a_{m,(k_m,\dots,k_1)}).
\]
\end{proof}
Now we are able -- by means of our homomorphisms $ \Psi $ and $ \Phi $ to transfer the algebraic results into the realm of weak approximation schemes.
\begin{defi}\label{k_norms}
  Let $k\in\N$ and let $g\in C_b^\infty(\R^d)$. Define
  \[
    \| g\|_k:= \sup_{i\leq k} \| \bigtriangledown^i g\|_\infty.
  \]
\end{defi}

\begin{rem}
The function $\| \cdot\|_k$ is a norm on $C_b^\infty(\R^d)$.
\end{rem}

\begin{defi}
Let $\mathcal{B}_k$ denote the space of bounded linear operators on $(C_b^\infty(\R^d),\|\cdot\|_k )$. We can regard $\mathcal{B}_k$ as a normed space with the operator norm. 
\end{defi}

\begin{prop}[\cite{IW}]
Fix a $k\in\N$. The following assertions hold:
\begin{enumerate}
  \item The family $(P_t)_{t\geq 0}$ is a uniformly bounded subset of $\mathcal{B}_k$.
  \item Let $\mathcal{A}$ be the geneartor of the diffusion process \eqref{sde} and let $N\in \N$. Then, for $g\in C_b^\infty(\R^d)$ we have
  \[
    (P_t g)(x)= \sum_{k=0}^N \frac{t^k}{k!}(\mathcal{A}^k g) (x) +\frac{1}{N!}\int_0^t (t-s)^N (P_s \mathcal{A}^{N+1}g)(x)q, ds.
  \]
\end{enumerate}
\end{prop}

We are now able to formulate weak approximation schemes of order $ 2m $ by means of our algebraic preparations. Recall therefore the definitions
\begin{align*}
\overrightarrow{Q_t^{[\theta]}} :&= \left( P^{(0)}_{t/\theta} \circ \cdots \circ P^{(d)}_{t/\theta}\right)^\theta, \\ 
\overleftarrow{Q_t^{[\theta]}} :&= \left( P^{(d)}_{t/\theta} \circ \cdots \circ P^{(0)}_{t/\theta}\right)^\theta, \\
Q_t^{[\theta]} :&= \frac{1}{2}(\overrightarrow{Q_t^{[\theta]}} + \overleftarrow{Q_t^{[\theta]}}). 
\end{align*}
of the building blocks of Ninomiya-Victoir schemes.
\begin{theo} \label{theo:final_result}
Let $ \{q_n\}_{n\in\mathbf{N}} $ be a generalized Fujiwara scheme of order $2m$, then 
$$ Q_{T,n} := \sum_{i=1}^mf_{\theta_i}(Q_{\frac{T}{n}}^{[\theta_i]})^n $$
for $ n \geq 0 $ is a scheme of weak approximation of order $ 2m $, where a choice of $ k $ is given by
$$
k = 2(2m + 1)(d+1) \sum_{i=1}^m \theta_i,
$$
that means
$$
| P_T \, g(x) - Q_{T,n} \, g(x) | \leq \frac{C}{n^{2m}} \| g \|_{k}
$$
for test functions $ g \in C^\infty_b (\mathbb{R}^N) $.
\end{theo}

\begin{proof}
Due to asymptotic formulas
$$
\Phi \Psi_{t} (j_{\leq 2m}(p)) = P_t + \mathcal{O}(t^{2m+1})
$$
and
$$
\Phi \Psi_{t} (j_{\leq 2m}(\exp(a_i))) = P^{(i)}_t + \mathcal{O}(t^{2m+1})
$$
where the constants in the Landau symbol depend on the derivatives of order at most $ 2(2m+1) $. Therefore we can simply copy the proof of Theorem \ref{theo:connection} by first replacing $ q_n $ with $ Q_{T,n} $ and $ p $ by $ P_T $. In the appearing sums we have to use the previous asymptotic formulas, namely
\begin{align*}
{\bigl( Q^{[\theta]}_{\frac{T}{n}} \bigr)}^n - P_{\frac{T}{n}} = & {\bigl( Q^{[\theta]}_{\frac{T}{n}} \bigr)}^n - \Phi \Psi_{\frac{T}{n}} (j_{\leq 2m}({(q^{[\theta]})}^n)) + \\
& + \Phi \Psi_{\frac{T}{n}} (j_{\leq 2m}({(q^{[\theta]})}^n - p)) + \\
& + \Phi \Psi_{\frac{T}{n}} (j_{\leq 2m}(p))-  P_{\frac{T}{n}},
\end{align*}
where the order behavior of the middle part has been shown in Theorem \ref{theo:connection} and the order behavior of the other two summands follows from the previous asymptotic formulas. Apparently each term in $ Q^{[\theta]} $, which is approximated due to the asymptotic formulas, increases the number of derivatives necessary to do the estimation by $ 2(2m+1) $, which leads to the formula for $ k $.
\end{proof}

\begin{ex}
The case $ m=1$ apparently corresponds to a version of the original Ninomiya-Victoir scheme.
\end{ex}

\begin{ex}
The case $ m=2$ corresponds to a scheme already presented in \cite{Fujiwara}. One can choose $ \theta_1 = 1 $ and $ \theta_2 = 2 $ and $ f_{\theta_1} = - \frac{1}{3} $ and $ f_{\theta_2} = \frac{4}{3}$.
\end{ex}

\begin{ex}
The case $ m=3 $ corresponds to Fujiwara's originally presented scheme, which in our language reads like follows. Notice that we do not need the full strength of our previous proof, which is built on Theorem \ref{theo:critical}.

For all mutually different numbers $ \theta_1, \theta_2, \theta_3 \in \mathbf{N} $, we can construct $6$-th order generalized Fujiwara scheme $ q $ with a form:
\begin{equation*}
q=f_{\theta_1}(q^{[\theta_1]})^n +f_{\theta_2}(q^{[\theta_2]})^n +f_{\theta_3}(q^{[\theta_3]})^n.
\end{equation*}

For the proof, which is presented for convenience here, we assume without loss of generality that $ \theta_1 < \theta_2 < \theta_3 $. We have,
\begin{equation*}
f=\begin{bmatrix}f_{\theta_1} \\ f_{\theta_2} \\ f_{\theta_3} \end{bmatrix} 
= \begin{bmatrix} \frac{\theta_1^4 }{(\theta_2^2 -\theta_1^2 ) (\theta_3^2 - \theta_1^2 )} \\
                  \frac{-\theta_2^4}{(\theta_3^2 - \theta_2^2 ) (\theta_2^2 - \theta_1^2 )} 
                  \\ \frac{\theta_3^4}{(\theta_3^2 - \theta_1^2 ) (\theta_3^2 - \theta_2^2 )}
  \end{bmatrix}.
\end{equation*}
By Corollary \ref{cor:order}, we have
\begin{align*}
j_{\le 4}(q^{[\theta_2]}-p) &= j_{\le 4}(\frac{\theta_1^2}{\theta_2^2}(q^{[\theta_1]}-p)), \\
j_{\le 4}(q^{[\theta_3]}-p) &= j_{\le 4}(\frac{\theta_1^2}{\theta_3^2}(q^{[\theta_1]}-p)).
\end{align*}
Then, by Corollary \ref{cor:substitute}, we have,
\begin{align*}
j_{\le 7}(q^{[\theta_2]}-p)^2) &= j_{\le 7}(\frac{\theta_1^4}{\theta_2^4}(q^{[\theta_1]}-p)^2), \\
j_{\le 7}((q^{[\theta_3]}-p)^2) &= j_{\le 7}(\frac{\theta_1^4}{\theta_3^4}(q^{[\theta_1]}-p)^2).
\end{align*}
Thus,
\begin{align*}
&j_{\le 7}(\sum_{i=1}^3f_{\theta_i}(q^{[\theta_i]} - p )^2) \\
&=j_{\le 7}((f_{\theta_1} + f_{\theta_2}\frac{\theta_1^4}{\theta_2^4} + f_{\theta_3}\frac{\theta_1^4}{\theta_3^4})(q^{[\theta_1]}-p)^2) \\
&=(\frac{\theta_1^4 }{(\theta_2^2 -\theta_1^2 ) (\theta_3^2 - \theta_1^2 )} - \frac{\theta_1^4}{(\theta_3^2 - \theta_2^2 ) (\theta_2^2 - \theta_1^2 )} + \frac{\theta_1^4}{(\theta_3^2 - \theta_1^2 ) (\theta_3^2 - \theta_2^2 )})j_{\le 7}((q^{[\theta_1]}-p)^2) \\
&=0.
\end{align*}
\end{ex}

\section[Implementation]{Implementation of a $m$--th order generalized Fujiwara scheme}

A scheme for approximation of expectation of order six was first introduced by Fujiwara \cite{Fujiwara}. In previous sections we theoretically constructed schemes for approximation of expectation of order $ 2m $ for arbitrary $ m\in\mathbf{N} $. In this section we show how to construct a practical scheme with approximating flow of vector fileds $ V_i $, which drive the SDE \eqref{sde}, by some suitable integration schemes. The usual choice for the integration schemes are Runge-Kutta methods. In our concrete example from mathematical finance we will use a seventh-order nine-stage explicit Runge-Kutta method with a very good stability, given by M.Tanaka et al. (see \cite{Tanaka1}, \cite{Tanaka2} and \cite{Tanaka3}). Higer order Runge-Kutta mehod often lose stability with respect to rounding error, truncated error and piling error. In addition, these effect decrease order of approximating error. Since in a concrete application of the algorithm, e.g. in mathematical finance, some of the ODEs can be very close to being stiff, the stability of the Runge--Kutta algorithm is of high importance. We show a relation between convergence order of weak approximation scheme and $m$-th order Runge-Kutta method. In addition we construct a concrete algorithm of a $m$-th order generalized Fujiwara scheme and analyze its computational cost and its approximating error. At the end we present a concrete numerical experiment. Tanaka's result is presented in the Appendix since we could not find any of his papers written in English.

The results of this section can be compared to those from \cite{ninnin:2009}.

\subsection{Runge-Kutta method}

For $ V \in C^\infty_b(\mathbf{R}^N,\mathbf{R}^N) $, the map $ \exp : C^\infty_b(\mathbf{R}^N,\mathbf{R}^N) \times \mathbf{R}_+ \times \mathbf{R}^N $ represents the flow driven by the vector field $V$ starting at $x_0$, i.e. the solution of the ordinary differential equation:
\begin{equation}\label{ode}
	\begin{aligned} 	
		\frac{d}{dt}x(t) &= V(x(t)), \\
		x(0) &= x_0.
	\end{aligned}
\end{equation}

\begin{defi}[$s$ stage explicit Runge--Kutta method of order $m$ for autonomous systems]
A $s$ stage explicit Runge--Kutta method of order $m$ for autonomous systems is determined by a lower triangular matrix $A=[a_{ij}]_{i,j=1}^s$ and a row $b=[ b_1 \cdots b_s]$ such that the following hold:
\begin{itemize}
 \item Let $h\in\mathbf{R}$, $t_0\in\mathbf{R}$ and let $t_n=t_{n-1}+h$ for all $n\in\mathbf{N}$. Given the vector $x_{n-1}$ as an approximation to $x(t_{n-1})$, where $x$ satisfies the equation \eqref{ode}, the approximation $x_n$ to $x(t_{n})$ is computed by evaluating, for $i=1,2,\dots,s$,
\[
  F_i=V(X_i),
\]
where $X_1,X_2,\dots,X_s$ are given by 
\[
 	X_i=x_{n-1}+h\sum_{j<i}a_{ij} F_j
\]
and then evaluating
\[
 	y_n=y_{n-1}+h\sum_{j=1}^s b_j F_j.
\]
\item The Taylor expansion of $x_n$ as a function of $h$ around $0$ should coincide with the Taylor expansion of $x(t_n)=x(t_{n-1}+h)$ up to (including) the term at the power $h^{m+1}$.
\end{itemize}
\end{defi}
\begin{rem}
 Usually Runge--Kutta methods are studied for general non-autonomous systems. In these cases the method is uniquely identified by a triplet $A$, $b$ and $c$, where $A$ and $b$ are as above and $c=[c_1 \dots c_s]^T$ is a suitable column vector.
\end{rem}
See Butcher \cite{butcher} and \cite{butcher2} for more details about the theory of Runge--Kutta method.


The next theorem shows that we need at least $ 12 $-th order Runge-Kutta method for $ 3 $-rd order generalized Fujiwara scheme.

\begin{theo}\label{thm:RK-order}
For all $ f \in C^\infty_{b}(\mathbf{R}^N\rightarrow \mathbf{R}^N) $, $ t \in \mathbf{R}_+ $ and $ x \in \mathbf{R}^N $, there exists $ C_i>0 $ such that
\begin{align*}
|f(\exp{(tV_0)}) - f(R_m(t,V_0)(x))| &\le C_0t^{m+1}, \\
|E[f(\exp{(( \sqrt{t}ZV_i)})-f(R_{2m}(\sqrt{t}Z,V_i)(x))]| &\le C_it^{m+1},
\end{align*}
where $ i \in \{1,\ldots ,d\} $ and $ Z \sim \mathcal{N}(0,1) $.
\end{theo} 
\begin{proof}
The first inequality follows from the definition of $m$-th order Runge-Kutta method and Taylor's theorem. Set $ i \in \{1,\ldots,d\} $. By the definition of Runge-Kutta method and Taylor's theorem again, we have,
\begin{align*}
&f(\exp{(( \sqrt{t}ZV_i)}) - f(R_{2m}(\sqrt{t}ZV_i)(x)) \\
&=\frac{t^{m+1/2}Z^{2m+1}}{2(m+1)!}V_i^{2m+1}f(x) + O(t^{m+1}). 
\end{align*}
Note that for all $ k \in \mathbf{N} $, $ E[Z^{2k+1}]=0 $ holds. Thus the conclusion is true.
\end{proof}

The next theorem shows that if we do not urge to have $O(n)$ computational cost, 4th order Runge-Kutta method is enough for sixth order scheme. 

\begin{theo}
For $ k,n \in \mathbf{N} $, for all $ f \in C_{b}^\infty(\mathbf{R}^N)$, for all $ i \in \{1,\ldots d\} $, and for all $ \ x \in \mathbf{R}^N $, there exists $ C_i>0 $ such that
\begin{equation*}
|E[f\left(\exp{\left(\frac{Z}{\sqrt{n}}V_i\right)}\right) - f\left( R_m\left( \frac{Z}{n^k\sqrt{n}},V_i\right)^{n^k}(x)\right) ]| \le \frac{C_i}{n^{km+k+m/2+1}}
\end{equation*}
holds where $ Z \sim \mathcal{N}(0,1) $.
\end{theo}
\begin{proof}
\begin{align*}
&E[f(\exp{(\frac{Z}{\sqrt{n}}V_i)(x)}) - f(R_m(\frac{Z}{n^k\sqrt{n}}{1},V_i)^{n^k}(x))]| \\
&= |E[(\exp{\frac{Z}{n^k\sqrt{n}}V_i})^{n^k}f(x) - R_m(\frac{Z}{n^k\sqrt{n}},V_i)^{n^k}f(x)]| \\
&=E[\sum_{l=0}^{n^k-1}(\exp{\frac{Z}{n^k\sqrt{n}}V_i})^l((\exp{\frac{Z}{n^k\sqrt{n}}V_i})- R_m(\frac{Z}{n^k\sqrt{n}},V_i)) \\
&R_m(\frac{Z}{n^l\sqrt{n}},V_i)^{n^k-l-1}f(x)] \\
&\le \frac{C_i}{n^{2(k+1/2)(m/2+1)}}n^k \\
&\le \frac{C_i}{n^{(km+k+m/2+1)}}.
\end{align*}
\end{proof}

\subsection[Recipe]{Recipe for $m$--th order generalized Fujiwara scheme}

In the following subsection we will provide the pseudocode for implementation of the $m$-th order generalized Fujiwara scheme with fixed coefficients $\theta_1, \theta_2, \dots, \theta_m$. Let $f=[f_1\cdots f_m]^T$ be as in the section \ref{sec:scheme} and let the function $solveDE(V,x_0,t)$ return the solution of the ODE \eqref{ode} at time $t$ with initial condition $x(0)=x_0$.

\begin{algorithm}
	\KwData{function $g$, vector fields $V_0,V_1,\dots,V_d$, time $T$, initial condition $x_0$, number of partition points $n$, number of samples $M$}
	\KwResult{approximation $E$ of the expectation $\mathrm{E}[f(X_T)]$, where $X_t$ is a process defined by the SDE \eqref{sde}}
	$Q \leftarrow 0\in \mathbf{R}^{1\times m}$\;
   \For(\tcc*[h]{expectation (MonteCarlo or quasi Monte Carlo)}){$o\leftarrow 1$ \KwTo $M$}{
   	$Q\leftarrow Q+samplePath(g,V_0,\dots,V_d,T,x_0,n)$\;
	}
   $Q\leftarrow \frac{1}{M}Q$\;
   \tcc{approx. for $E(g(X(T,x_0)))$ is the linear combination $\sum_i f_ i* Q_i$}
   $E\leftarrow Q\, f$\;
   \Return{$E$}
	\caption{Fujiwara}\label{alg:fujiwara}
\end{algorithm}

\begin{algorithm}[H]
	\KwData{function $g$, vector fields $V_0,V_1,\dots,V_d$, time $T$, initial condition $x_0$, number of partition points $n$}
	\KwResult{ row vector $Q=[Q_s^{[\theta_1]}\cdots Q_s^{[\theta_m]}]\in\mathbf{R}^{1\times m}$ calculated for a random simulated path}
	$Q\leftarrow 0\in\mathbf{R}^{1 \times m}$\;
    generate independent Bernoulli(1/2) random variables $\Lambda= [\Lambda_1,\dots,\Lambda_n]$\;
    \For {$k\leftarrow 1$ \KwTo $m$}{
      \tcc{generate all standard normal random variables that are needed}
      $D\leftarrow T/(n\theta_k) [1\dots 1]\in \mathbf{R}^{1\times (\theta_k n)}$\;
      $N\leftarrow \sqrt{T/(n \theta_k)} [N_1,\dots, N_{\theta_k n} ]$, where $N_1,\dots, N_{\theta_k n} $ are i.i.d., with $N_1\sim N(0,I_d)$\; 
      \tcc{first row serves for the component without Brownian motion}
      $Z\leftarrow \begin{bmatrix}
                   D \\
                   N
                  \end{bmatrix}$\;
      $X\leftarrow x_0$\; 
      \For(\tcc*[f]{consequtively solve the ODE's}){$j\leftarrow 1$ \KwTo $n$}{ 
        \eIf(\tcc*[f]{solve appropriate ODE}){ $\Lambda_j=1$}{
          \For(\tcc*[f]{repetition because of finer dissection}) {$\Theta\leftarrow 1$ \KwTo $\theta_k$}{
            \For(\tcc*[f]{solving ODE's}) {$i\leftarrow 0$ \KwTo $d$}{
              $X\leftarrow solveDE(V_i,X,Z_{i+1,(\Theta-1)*n+j})$\;
              }
            }
          }
          {
          \For(\tcc*[f]{repetition because of finer dissection}) {$\Theta\leftarrow 1$ \KwTo $\theta_k$}{
            \For(\tcc*[f]{solving ODE's}) {$i\leftarrow 0$ \KwTo $d$}{
              $X\leftarrow solveDE(V_{d-i},X,Z_{d-i+1,(\Theta-1)*n+j})$\;
            }
          }
        }
      }
      $Q_s^{[\theta_k]}\leftarrow g(X)$\;
    }	
	\Return{$Q$}
	\caption{samplePath}\label{alg:sample}
\end{algorithm}

\begin{rem}
 Usually in modern computers memory size is no longer an issue. From this perspective it seems sensible to generate all needed random variables in advance. Namely, the random variables for various $\theta_i$'s do not have to be independent, therefore we can reduce its number by reusing them, and there exist efficient algorithms which speed up the process of their generation if we do it in one batch instead of step by step as it is written in Algorithm \ref{alg:sample}.
\end{rem}

\subsection{Computational cost}

\begin{theo}\label{thm:computational_cost}
Let $ d,n,M,m, T, \theta_1,\dots,\theta_m$ be as above, such that $ T/n $ is sufficiently small. Furthermore, assume that each step of the method $solveDE$ needs $a$ operations, i.e. additions, multiplications and function evaluations, that $B$ operations are needed to generate a (pseudo or quasi) Bernoulli random variable and that $Z$ operations are needed to generate a standard $d$-dimensional normally distributed (pseudo or quasi) random variable. Then the computational cost of Algorithm \ref{alg:fujiwara} is $ M\Big(5m+n\big((d+1)a+Z+1\big)\sum_{k=1}^m \theta_k +nB+1\Big) +2m $.
\end{theo}
\begin{proof}
Let us denote the computational cost of the Algorithm \ref{alg:sample} by $C$. A straightforward calculation shows that the computational cost of the Algorithm \ref{alg:fujiwara} is equal to $M(C+1)+2m$.

For fixed $j\in\{1,\dots,n\}$ in Algorithm \ref{alg:sample} we have $\theta_k (d+1)a$ operations. Hence, for fixed $k\in\{1,\dots,m\}$ there are $5+n\theta_k((d+1)a+Z+1)$ operations. It follows that $C=5m+n\big((d+1)a+Z+1\big)\sum_{k=1}^m \theta_k +nB$.
\end{proof}

\begin{rem}
 Rigorous use of Runge-Kutta algorithms for solving ODEs in the algorithm is only suitable for building a universal solver for SDEs of the type \eqref{sde}. In concrete practical applications it is to be expected that many of the ODEs of the type \eqref{ode} have a nice enough explicit solution 
\end{rem}

The error of the algorithm consists of discretization part, i.e. the error due to numerical solution of ODEs and the error which comes from the scheme, and of the convergence error which comes from the Monte Carlo or quasi Monte Carlo simulation. 

\begin{theo}
For $ n,M \in \mathbf{N} $ such that $ T/n $ is sufficiently small. The approximation error of Algorithm 1 is $ O(1/n^{2m}) +  O(1/\sqrt{M}) $.
\end{theo}

\begin{rem}
 One should take great care when choosing a suitable subdivision of the interval, since the coefficient of the discretisation error directly depends on function $f$ and vector fields $V_i$, thus, although bounded, the coefficient can get fairly large in some cases. Moreover, the convergence error of the Monte Carlo simulation is directly proportional to the sqare root of variance of $f(X(T,x))$. As in the case of discretisation error this should be taken into account, since, although constant, the variance can be large comparing to the size of error we would like to achieve.
\end{rem}

\subsection{Numerical example}

For our numerical example we have chosen the genearlized Fujiwara scheme of order $6$ with $\theta_1=1$, $\theta_2=2$ and $\theta_3=3$, i.e. the scheme that first appeared in \cite{Fujiwara}, and the generalized Fujiwara scheme of order $8$ with the choice of parameters $\theta_1=1$, $\theta_2=2$, $\theta_3=3$ and $\theta_4=4$.

In order to compare the algorithm to the basic Ninomiya-Victoir scheme we consider an Asian call option written on an asset whose price process follows the Heston stochastic volatility model. Let $X_1$ be the price process of an asset following the Heston model:
\begin{equation}\label{eq:heston1}
	\begin{aligned}
	 	X_1(t,x) =x_1 &+\int_0^t \mu X_1(s,x)\, ds + \int_0^t X_1(s,x) \sqrt{X_2(s,t)}\, dB^1(s)\\
	 	X_2(t,x) =x_2 &+\int_0^t \alpha (\theta-X_2(s,x))\, ds\\
	 					& + \int_0^t \beta \sqrt{X_2(s,t)}\big(\rho \, dB^1(s) +\sqrt{1-\rho^2}\, dB^2(s)\big),
	\end{aligned}
\end{equation}
where $x=(x_1,x_2)\in (\mathbf{R}_{>0})^2$, $(B^1(t),B^2(t))$ is a two-dimensional standard Brownian motion, $-1\leq \rho\leq 1$ and $\alpha,\theta,\mu$ are some positive coefficients satisfying $2\alpha\theta -\beta^2>0$ to ensure that the volatility does not reach zero. The payoff of the Asian call option on this asset with maturity $T$ and strike $K$ is $\max(X_3(T,x)/T-K,0)$, where
\begin{equation}\label{eq:heston2}
	X_3(t,s)=\int_0^t X_1(s,x)\, ds.
\end{equation}
Hence, the price of this option becomes $D\times E[\max(X_3(T,x)/T-K,0)]$ where $D$ is an appropriate discount factor on which we do not focus in this experiment. As in \cite{NV} take $T=1$, $K=1.05$, $\mu=0.05$, $\alpha=2.0$, $\beta=0.1$, $\theta=0.09$, $\rho=0$ and $x=(1.0,0.09)$. 

Up to the error of the magnitude $10^{-6}$ we have
\[
 	E[\max(X_3(T,x)/T-K,0)]=6.0473534496*10^{-2}
\]
obtained from \cite{NN}. Let $X(t,x)=(X_1(t,x),X_2(t,x),X_3(t,x))^T$. SDEs \eqref{eq:heston1} and \eqref{eq:heston2} can be transformed in the Stratonovich form since $ X_2 \neq 0 $:
\[
 X(t,x)=\sum_{i=0}^2 \int_0^t V_i(X(s,x))\circ dB^i(s),
\]
where
\begin{equation}\label{eq:vector_fields}
 	\begin{aligned}
 	  V_0(y_1,y_2,y_3)&=(y_1(\mu-\frac{y_2}{2}-\frac{\rho\beta}{4}), \alpha(\theta-y_2)-\frac{\beta^2}{4}, y_1)^T\\
 	  V_1(y_1,y_2,y_3)&=(y_1\sqrt{y_2},\rho\beta\sqrt{y_2},0)^T\\
 	  V_2(y_1,y_2,y_3)&=(0, \beta\sqrt{(1-\rho^2)y_2}, 0)^T.
 	\end{aligned}
\end{equation}

Taking our choice of $\rho=0$ into consideration we get exact solutions of ODEs of the type \eqref{ode} driven by vector fields $V_1$ and $V_2$ (see \cite{NV} for more details):
\begin{equation}\label{eq:Heston_solution}
 	\begin{aligned}
 	 	&\exp(tV_1)(x_1,x_2,x_3)^T=(x_1 e^{t\sqrt{x_2}},x_2,x_3),\\
 	 	&\exp(tV_2)(x_1,x_2,x_3)^T=\Big(x_1,\big( \frac{\beta t}{2}+\sqrt{x_2}\big)^2,x_3\Big).
 	\end{aligned}
\end{equation}

According to the proof of Theorem \ref{thm:RK-order} we need a Runge--Kutta
method of order at least $6$ to approximate the solution
$\exp(tV_0)(x_1,x_2,x_3)^T$ for generalized Fujiwara scheme of order $6$ and a
Runge--Kutta method of order at least $8$ for a generalized Fujiwara scheme of
order $8$ if we want a linear algorithm. If we allow quadratic computational
cost for the the generalized Fujiwara scheme of the weak order $8$, it is
sufficient to use a Runge--Kutta method of order $4$. In our example we used $9$
stage $7$-th order Runge--Kutta method from \cite{Tanaka1}, defined by the Butcher's tableau
presented in the Appendix.

The pseudorandom numbers in MC were generated by the Mersenne twister algorithm. The QMC was performed using Sobol sequence, generated by the library SobolSeq51.dll provided by Broda (see \cite{broda}). Both MC and QMC integration were performed using $10^8$ sample paths.

The use of exact solutions of ODEs driven by vector fields $V_1$ and $V_2$ reduces the computational cost of the algorithm by $ 2 M na\sum_{k=1}^o \theta_k  $, where $o$ designates the order of weak generalized Fujiwara scheme divided by $2$, $M$ denotes the number of MC/QMC sample paths, $n$ is the number of subdivision points and $a$ is the number of operations required for solving ODE's driven by $V_1$ or $V_2$, if we compare it to the results of Theorem \ref{thm:computational_cost}.

\begin{center}
\begin{tiny} 
\begin{tabular}{lrrrr} \toprule
Method / $ n $ & $ 2 $ & $3$ & $4$ & $5$\\ \midrule
NV & 0.00208536744970740 & 0.00095536839891733   & 0.00055694952858933 \\
GF (order 6) MC &  0.00006154245956983  &  0.00003651735446759  &  0.00003522768790512 \\
GF (order 6) QMC & 0.00005526280089 &  0.0000105789197729  &  0.0000040357269938  &  0.0000028986604713 \\
GF (order 8) MC &  0.00004536485526115  &  0.00003694928288030  & 0.000055051968504230\\ 
GF (order 8) QMC & 0.0000178413262662  &  0.0000013695959963  &  0.0000010913411477 \\ \bottomrule
\end{tabular}
\end{tiny}
\end{center}

\begin{figure}[ht]
\includegraphics[scale=0.6]{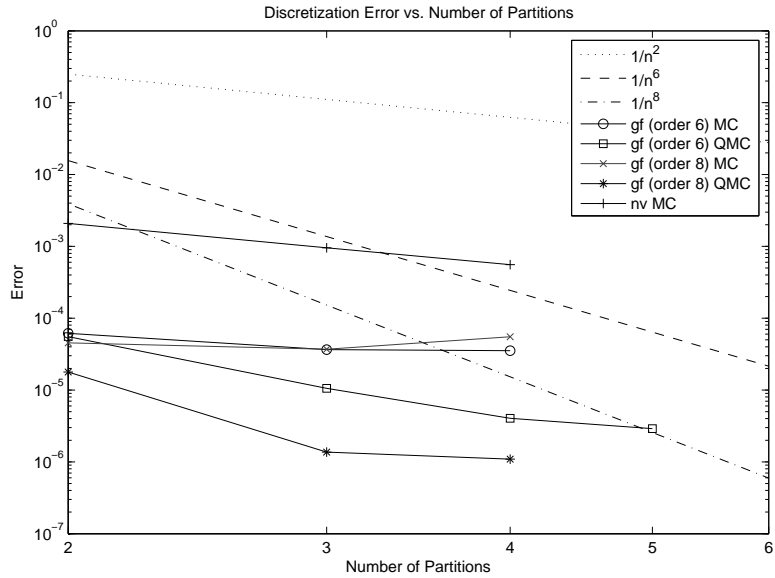}
\caption{Error coming from discretization}\label{fig:disc-error}
\end{figure}

\begin{figure}[ht]
 \includegraphics[scale=0.6]{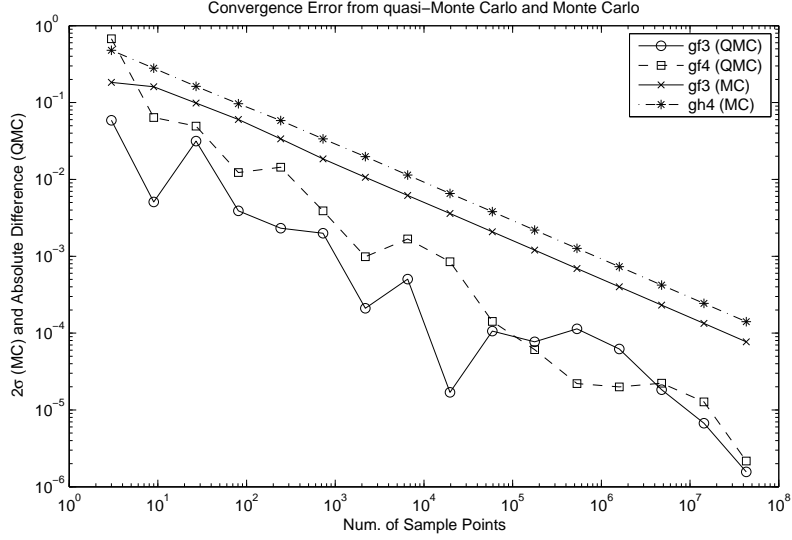}
\caption{Error coming from integration}
\label{fig:int-error}
\end{figure}

The graph in Fig. \ref{fig:disc-error} clearly shows that the new extrapolation method
reduces the order of the discretization error in comparison to the original Ninomiya-Victoir
algorithm for several magnitudes. In the MC case the discretization error almost immediatly converges to the integration error (see Fig \ref{fig:disc-error} and Fig. \ref{fig:int-error}). Also in the QMC case the discretization
error is soon (for small $n$) overshadowed by the integration error caused by QMC integration (see Fig. \ref{fig:int-error}), the weak order of the extrapolated algorithms can still be observed from the slope of curves in the graph in Fig. \ref{fig:disc-error}.

\section*{Acknowledgements}

The research was done while first and third authors were guests at the Research Unit of Financial and Actuarial Mathematics at the Vienna University of Technology. The research of the third author was supported by AMaMeF exchange grant No.~2080. A part of the research of the second and the third author was kindly supported by Vienna WWTF project ``Mathematik und Kreditrisiken''. The second author was kindly supported by means of the START prize project Y 328.

\section{Appendix}

We give our original proof of Lemma \ref{lem:fujiwara}.

\begin{defi}
Let $\mathfrak{g}$ be a Lie algebra. For $X, Y \in \mathfrak{g}$ define $c_1 (X,Y)= X+Y$ and $c_n(X,Y)$ by the following recursion formula
\begin{align*}
(n+1) & c_{n+1} (X,Y) = \frac{1}{2} [X-Y,c_n(X,Y)]+ \\
& +\sum_{p\geq 1, 2p\leq n} K_{2p} \sum_{\substack{k_1,\dots,k_{2p}\geq 0,\\ k_1+\dots+k_{2p}=n}} [c_{k_1} (X,Y),[\dots,[c_{k_{2p}}(X,Y),X+Y]\dots]],
\end{align*}
where $K_{2p}$ are coefficients defined in \cite[2.15.9]{Var}
\end{defi}
For more details about $c_n(X,Y)$ see \cite[Sec. 2.15]{Var}.

\begin{lem}\label{lem:antisim}
\begin{equation}\label{eq:exchange}
c_n(X,Y)=(-1)^{n+1} c_n(Y,X)
\end{equation}
\end{lem}
\begin{proof} For $n=1$ the assertion is clear.

Suppose we have $c_m(X,Y)=(-1)^{m+1} c_m(Y,X)$ for all $m\leq n$. By recursion we obtain
\begin{align*}
&(n+1)c_{n+1} (Y,X) \\
&= \frac{1}{2} [Y-X,c_n(Y,X)]+ \\
& +\sum_{p\geq 1, 2p\leq n} K_{2p} \sum_{\substack{k_1,\dots ,k_{2p}\geq 0, \\ k_1+\dots +k_{2p}=n}} [c_{k_1} (Y,X),[\dots,[c_{k_{2p}}(Y,X),X+Y]\dots ]].
\end{align*}
Using the induction hypothesis and bilinearity of Lie brackets, the above equation transforms into
\begin{align*}
&(n+1) c_{n+1} (Y,X) \\
&= \frac{1}{2} (-1)^{n+2}[X-Y,c_n(X,Y)]+ \\
& +\sum_{p\geq 1, 2p\leq n} K_{2p} \sum_{\substack{k_1,\dots,k_{2p}\geq 0,\\ k_1+\dots +k_{2p}=n}} (-1)^{k_1+\dots k_{2p}+2p}[c_{k_1} (X,Y),[\dots, \\
&[c_{k_{2p}}(X,Y),X+Y]\dots ]]\\
& = (-1)^{n+2} \Big(\frac{1}{2} [X-Y,c_n(X,Y)]+ \\
& +\sum_{p\geq 1, 2p\leq n} K_{2p} \sum_{\substack{k_1,\dots,k_{2p}\geq 0,\\ k_1+\dots +k_{2p}=n}} [c_{k_1} (X,Y),[\dots,[c_{k_{2p}}(X,Y),X+Y]\dots ]]\Big)\\
&= (-1)^{n+2} (n+1)c_{n+1}(X,Y)
\end{align*}which proves the assertion.
\end{proof}

Let $\tau_{i,d}$ denote $j_i(\overrightarrow{q^{[1]}})$.

\begin{proof}[Proof of Lemma \ref{lem:fujiwara}] The case $d=0$ is trivial. Next we consider the case $d=1$. Using Baker--Campbell--Hausdorff formula to expand $\tau_{l,1}$ and $j_l(\log(\overleftarrow{q^{[1]}}(1)))$ and applying \eqref{eq:exchange} proves the formula \eqref{eq:expansion}. By applying Baker--Campbell--Hausdorff formula to the definition of $\tau_{l,d}$ we get 
\begin{align*}
\tau_{l,d}&=j_l(\log(\overrightarrow{q^{[1]}}(d)))=j_l\big(\log(\exp(\log(\overrightarrow{q^{[1]}}(d-1)))\exp(a_d))\big)\\ &= j_l\big(\sum_{k=1}^l c_k(\sum_{j=1}^l \tau_{j,d-1},a_d)\big).
\end{align*}
Suppose that for all $n\in\mathbf{N}$, $n<d$ we have
\[
\log(\overleftarrow{q^{[1]}}(n))=\sum_{i=1}^\infty (-1)^{i+1} \tau_{i,n}.
\]
Using Lemma \ref{lem:antisim}, the induction hypothesis and the BCH-formula on $j_l(\log(\overleftarrow{q^{[1]}}(d)))$ gives us
\begin{align*}
j_l(\log(\overleftarrow{q^{[1]}}(d)))&=j_l\Big(\log\big(\exp(a_d) \exp(\log(\overleftarrow{q^{[1]}}(d-1)))\big)\Big)\\ &=j_l\big(\sum_{k=1}^l c_k(A_d, \log(\overleftarrow{q^{[1]}}(d-1)))\big)\\ &=j_l\big(\sum_{k=1}^l c_k(a_d, \sum_{j=1}^l (-1)^{j+1} \tau_{j,d-1})\big)\\ &=j_l\big(\sum_{k=1}^l (-1)^{k+1}c_k(\sum_{j=1}^l (-1)^{j+1} \tau_{j,d-1},a_d)\big).
\end{align*}
Thus, it is sufficient to show that for all $k\in\{1,\dots,l\}$ and $l\in\mathbf{N}$  we have
\begin{equation}\label{eq:hipoteza}
j_l\big(c_k(\sum_{j=1}^l (-1)^{j+1} \tau_{j,d-1},a_d)\big)=(-1)^{k+l} j_l\big(c_k(\sum_{j=1}^l \tau_{j,d-1},a_d )\big).
\end{equation}
Note that the equality in \eqref{eq:hipoteza} holds trivially for $k>l$.

Since $\tau_{j,d-1}$ is a homogeneous polynomial of degree $j$, the assertion is clear for $k=1$ and all $l\in \mathbf{N}$.  It is easy to see that for $l'<l$ we have
\begin{equation}\label{eq:skrcitev}
j_{l'}\big(c_m(\sum_{j=1}^l (-1)^{j+1} \tau_{j,d-1},a_d)\big)= j_{l'}\big(c_m(\sum_{j=1}^{l'} (-1)^{j+1} \tau_{j,d-1},a_d)\big).
\end{equation}
Let now
$j_l\big(c_m(\sum_{j=1}^l (-1)^{j+1} \tau_{j,d-1},a_d)\big)=(-1)^{k+l} j_l\big(c_m(\sum_{j=1}^l \tau_{j,d-1},a_d )\big)$ for all $m\in\{1,\dots,k\}$ and $l\in\mathbf{N}$, then we have
\begin{align*}
&j_l\big(c_{k+1}(\sum_{j=1}^l (-1)^{j+1} \tau_{j,d-1},a_d)\big) \\
&= \frac{1}{k+1}j_l \Big( \frac{1}{2} \big[ \sum_{j=1}^l (-1)^{j+1} \tau_{j,d-1}-a_d, c_k(\sum_{j=1}^l (-1)^{j+1} \tau_{j,d-1},a_d)\big] \\ 
&+ \sum_{\substack{p\geq 1\\ 2p\leq k+1}} K_{2p} \sum_{\substack{k_1,\dots,k_{2p}>0\\ k_1+\dots+k_{2p}=k+1}} \big[ c_{k_1}(\sum_{j=1}^l (-1)^{j+1} \tau_{j,d-1},a_d),\\
&\big[c_{k_2}(\sum_{j=1}^l (-1)^{j+1} \tau_{j,d-1},a_d),\dots,\\
&\big[c_{k_{2p}}(\sum_{j=1}^l (-1)^{j+1} \tau_{j,d-1},a_d),\sum_{j=1}^l (-1)^{j+1} \tau_{j,d-1}+a_d\big]\dots\big]\big]\Big)=   \\
&=\frac{1}{k+1} \Big( \frac{1}{2} \sum_{j=1}^{l-1} \big[ (-1)^{j+1} \tau_{j,d-1}, j_{l-j}(c_k(\sum_{j=1}^l (-1)^{j+1} \tau_{j,d-1},a_d))\big]+\\ 
&+ \big[ -a_d, j_{l-1}(c_k(\sum_{j=1}^l (-1)^{j+1} \tau_{j,d-1},a_d))\big]+\\ &+ \sum_{\substack{p\geq 1\\ 2p\leq k+1}} K_{2p} \sum_{\substack{k_1,\dots,k_{2p}>0\\ k_1+\dots+k_{2p}=k}} \sum_{\substack{m_1,\dots,m_{2p+1}>0\\ m_1+\dots+m_{2p+1}=l}} \big[ j_{m_1}(c_{k_1}(\sum_{j=1}^l (-1)^{j+1} \tau_{j,d-1},a_d)),\\ 
&\big[j_{m_2}(c_{k_2}(\sum_{j=1}^l (-1)^{j+1} \tau_{j,d-1},a_d)),\dots, \big[j_{m_{2p}}(c_{k_{2p}}(\sum_{j=1}^l (-1)^{j+1} \tau_{j,d-1},a_d)),\\
&(-1)^{m_{2p+1}+1} \tau_{m_{2p+1},d-1}+j_{m_{2p+1}}(a_d)\big]\dots\big]\big]\Big).
\end{align*}
Using \eqref{eq:skrcitev}, the induction hypothesis and the bilinearity of Lie brackets the above expression transforms into
\begin{align*}
j_l\big(c_{k+1}&(\sum_{j=1}^l (-1)^{j+1} \tau_{j,d-1},a_d)\big)= \frac{1}{k+1} \Big( \frac{1}{2} \sum_{j=1}^{l-1} (-1)^{j+1+l-j+k} \big[ \tau_{j,d-1},\\ 
&j_{l-j}(c_k(\sum_{j=1}^l  \tau_{j,d-1},a_d))\big] + (-1)^{l+k-1}\big[ -a_d, j_{l-1}(c_k(\sum_{j=1}^l  \tau_{j,d-1},a_d))\big]\\ 
&+ \sum_{\substack{p\geq 1\\ 2p\leq k+1}} K_{2p} \sum_{\substack{k_1,\dots,k_{2p}>0\\ k_1+\dots+k_{2p}=k}} \sum_{\substack{m_1,\dots,m_{2p+1}>0\\ m_1+\dots+m_{2p+1}=l}} (-1)^{m_1+\dots+m_{2p+1}+1+k_1+\dots+k_{2p}}\\ 
&\big[ j_{m_1}(c_{k_1}(\sum_{j=1}^l \tau_{j,d-1},a_d)),\big[j_{m_2}(c_{k_2}(\sum_{j=1}^l  \tau_{j,d-1},a_d)),\dots,\\ &\big[j_{m_{2p}}(c_{k_{2p}}(\sum_{j=1}^l  \tau_{j,d-1},a_d)), \tau_{m_{2p+1},d-1}+j_{m_{2p+1}}(a_d)\big]\dots\big]\big]\Big).
\end{align*}
Thus,
\begin{align*}
&j_l\big(c_{k+1}(\sum_{j=1}^l(-1)^{j+1} \tau_{j,d-1},a_d)\big)\\
&= (-1)^{k+l+1} \frac{1}{k+1}j_l \Big( \frac{1}{2} \big[ \sum_{j=1}^l \tau_{j,d-1}-a_d, c_k(\sum_{j=1}^l  \tau_{j,d-1},a_d)\big] \\ 
&+ \sum_{\substack{p\geq 1\\ 2p\leq k+1}} K_{2p} \sum_{\substack{k_1,\dots,k_{2p}>0\\ k_1+\dots+k_{2p}=k+1}} \big[ c_{k_1}(\sum_{j=1}^l  \tau_{j,d-1},a_d),\big[c_{k_2}(\sum_{j=1}^l \tau_{j,d-1},a_d),\dots,\\ 
&\big[c_{k_{2p}}(\sum_{j=1}^l \tau_{j,d-1},a_d),\sum_{j=1}^l  \tau_{j,d-1}+a_d\big]\dots\big]\big]\Big) \\
&= (-1)^{k+l+1}j_l\big(c_{k+1}(\sum_{j=1}^l \tau_{j,d-1},a_d )\big),
\end{align*}
which is the desired result.
\end{proof}


\begin{thebibliography}{9}
   \bibitem{broda} British-Russian Offshore Development Agency (BRODA), \texttt{http://www.broda.co.uk/}. 
	\bibitem{butcher} Butcher, J. C., The Numerical Analysis of Ordinary Differential Equations, John Wiley \& Sons, Chichester--New York--Brisbane--Toronto--Singapore, 1987.
	\bibitem{butcher2} Butcher, J. C., Numerical Methods for Ordinary Differential Equations, John Wiley \& Sons, Chichester, 2003. 
	\bibitem{cohn} Cohn, P. M., Skew fields. Theory
of general division rings, Encyclopedia of Mathematics and its Applications 57, Cambridge
University Press, Cambridge, 1995.
  \bibitem{Fujiwara} Takehiro Fujiwara, Sixth order methods of Kusuoka approximation, UTMS 2006-7.
\bibitem{IW} Nobuyuki Ikeda and Shinzo Watanabe, Stochastic Differential Equations and Diffusion Processes, Second Edition, Elsevier,1989. 
\bibitem{Kohatsu} Arturo Kohatsu-Higa , Weak approximations. A Malliavin calculus approach, Math. Comp., \textbf{70}(2001), 135-172.
	\bibitem{NN} Mariko Ninomiya and Syoiti Ninomiya, A new weak approximation scheme of stochastic differential equations and the Runge--Kutta mathod, arXiv:0709.2434v3.
  \bibitem{NV} Syoiti Ninomiya and Nicolas Victoir, Weak approximation of stochastic differential equations and application to derivative pricing,  Appl. Math. Finance 15, no. 1-2 (2008), 107--121. arXiv:math/0605361.
\bibitem{ninnin:2009} Mariko Ninomiya and Syoiti Ninomiya, A new higher-order weak approximation scheme for stochastic differential equations and the Runge-Kutta method. Finance Stoch. 13 (2009), no. 3, 415--443.
\bibitem{Tanaka1} Tanaka, M., Kasahara, E., Muramatsu, S. and Yamashita, S., 
On a Solution of the Order Conditions for the Nine-Stage Seventh-Order Explicit Runge-Kutta Method(in Japanese), 
Information Processing Society of Japan, Vol. 33, No. 12(1992) 1506-1511.
\bibitem{Tanaka2} Tanaka, M., Muramatsu, S. and Yamashita, S., 
On the Optimization of Some Nine-Stage Seventh-Order Runge-Kutta Method(in Japanese), 
Information Processing Society of Japan, Vol. 33, No. 12(1992) 1512-1526.
\bibitem{Tanaka3} Tanaka, M., Yamashita, S., Kubo, E. and Nozaki, Y., 
On Seventh-order Nine-stage Explicit Runge-Kutta Methods with Extended Region of Stability(in Japanese), 
Information Processing Society of Japan, Vol. 34, No. 1(1993) 52-61.
 \bibitem{Var} V. S. Varadarajan, Lie groups, Lie algebras, and their representations, Springer--Verlag, New York, 1984.
\end{thebibliography}
\end{document}